\documentclass[11pt,twoside]{article}
\usepackage{latexsym}
\usepackage{amssymb,amsbsy,amsmath,amsfonts,amssymb,amscd}
\usepackage{graphicx,color}
\usepackage{float,url}
\usepackage{cancel}
\usepackage[colorlinks,linkcolor=red,citecolor=blue]{hyperref}
\newcommand  \ind[1]  {   {1\hspace{-1.2mm}{\rm I}}_{\{#1\} }    }
\newcommand {\wt}[1] {{\widetilde #1}}
\newcommand{\commentout}[1]{}
\newcommand{\R}{\mathbb{R}}

\newcommand {\eps}  {\varepsilon}

\newcommand {\sg} {\sigma}
\newcommand {\vp} {\varphi}
\newcommand {\lb} {\lambda}

\newcommand {\cam} { {\mathcal M} }

\newcommand {\f}   {\frac}
\newcommand {\p}   {\partial}
\newcommand{\red}{\textcolor{red}}

\newcommand{\beq}{\begin{equation}}
\newcommand{\eeq}{\end{equation}}
\newtheorem{theorem}{Theorem}
\newtheorem{lemma}[theorem]{Lemma}

\newtheorem{remark}[theorem]{Remark}

\newcommand{\qed}{{ \hfill
                       {\unskip\kern 6pt\penalty 500 \raise -2pt\hbox{\vrule\vbox to 6pt{\hrule width 6pt
                       \vfill\hrule}\vrule} \par}   }}
\title{Collective motion driven by nutrient consumption}

\author{Pierre-Emmanuel Jabin\thanks{P.--E. Jabin.  Department of Mathematics and Huck Institutes, Pennsylvania State University, State College, PA 16801, USA. Email: pejabin@psu.edu} \and
Beno\^ \i t Perthame\thanks{Sorbonne Universit{\'e}, CNRS, Universit\'{e} de Paris, Inria, Laboratoire Jacques-Louis Lions UMR7598, F-75005 Paris. 
Email : Benoit.Perthame@sorbonne-universite.fr}
\thanks{B.P. has received funding from the European Research Council (ERC) under the European Union's Horizon 2020 research and innovation programme (grant agreement No 740623). P.E.J.  is partially supported by NSF DMS Grants DMS-2049020, DMS-2205694, and DMS-2219397.}
}
\date{\today}

\begin{document}
\maketitle
\pagestyle{plain}
\pagenumbering{arabic}

\begin{abstract} 
A classical problem describing the collective motion of cells, is the movement driven by consumption/depletion of a nutrient. Here we analyze one of the simplest such model written as a coupled Partial Differential Equation/Ordinary Differential Equation system which we scale so as to get a limit describing the usually observed pattern. In this limit the cell density is concentrated as a moving Dirac mass and the nutrient undergoes a discontinuity. 

We first carry out the analysis without diffusion, getting a complete description of the unique  limit. When diffusion is included, we prove several specific a priori estimates and interpret the system as a heterogeneous monostable equation. This allow us to obtain a limiting problem which shows the concentration effect of the limiting dynamics. 
\end{abstract} 
\vskip .7cm

\noindent{\makebox[1in]\hrulefill}\newline
\section{Introduction}
\label{sec:intro}

A classical problem describing the collective motion of cells, is the movement driven by consumption/depletion of a nutrient~\cite{BenJacob, Mimura, Murray2}.
The simplest description uses a number density of cells $u_\eps$ and a nutrient concentration $v_\eps$. It is written
\beq \begin{cases}
\p_t u_\eps - \eps \p^2_{xx} u_\eps = \f 1 \eps u_\eps (v_\eps- \mu),  \qquad t \geq 0, \; x \in \R,
\\
\p_t v_\eps = - u_\eps v_\eps,
\end{cases}
\label{eq:full}
\eeq
completed with initial data $u_\eps^0 $, $v_\eps^0 $, such that 
\[
\eps u_\eps^0 \in L^1_+(\R) , \qquad  0 < v_m \leq v_\eps^0  \leq v_M < \infty, \qquad   v_m < \mu < v_M.
\]
We have introduced a parameter $\eps$ which measures the time scale of the cell random motion compared to nutrient consumption. Our interest here is when this parameter is small because it is a case when a pattern is produced under the form of a high concentration of cells despite the parabolic character of Equation~\eqref{eq:full}. In fact this phenomena is closely related to concentration effects in non-local semi-linear parabolic equations as studied intensively recently, see \cite{Dieketal2005, BarPer2008,CJ2011,LorMirPer2011} and the references therein. This analogy leads us to postulate that $u_\eps$ concentrate as Dirac masses at points where $v_\eps$ undergoes a discontinuity. 
\\
The scale proposed here, which is chosen to produce a distinguished limit, is usual for semi-linear diffusion equations and has been studied for local problems in classical works, \cite{FSoug86,BES90}. The most efficient method is to use the Hopf-Cole transform and viscosity solutions of Hamilton-Jacobi equations \cite{CL_viscosity,CEL_hj}. We restrict our analysis to one dimension to explain the solution structure as depicted in Figure~\ref{fig:nutrientbact} but significant parts of our analysis can be extended to several dimensions.

Several related studies can be mentioned. Coupling Ordinary and  Partial Differential Equation is rather classical in different areas: for pattern formation, see \cite{marciniak20} (study of existence and stability of stationary solutions), modeling of two species dynamics with an unmotile specie \cite{chauviere12,SW21} for instance.  Traveling waves with a non-motile phase have also been studied, see~\cite{ZhangZhao07} and the references therein. However we are not aware of any analytical study related to the scaling proposed here.

\begin{figure}[!h]
    \centering
    \includegraphics[width=.4\textwidth]{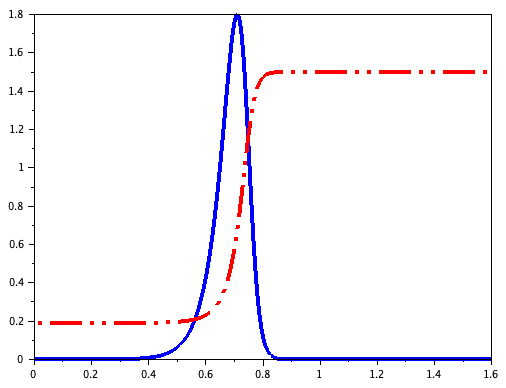}
    \\[-5pt]
   \caption{ Traveling wave solution of Equation~\eqref{eq:full}. 
   In blue/solid line the component $u_\eps$. In red/dashed line, the nutrient $v_\eps$.}
   \label{fig:nutrientbact}
\end{figure}

Our approach combines a reformulation of the problem which leads to a  heterogeneous monostable equation for $v_\eps$ and the standard Hopf-Cole transform  as mentioned above. This a convenient tool to represent the Dirac concentration of $u_\eps$ under the form $\exp (- \vp(t,x)/\eps )$ where $\vp$ behaves like a quadratic function.
\\

In Section \ref{sec:nodiff}, we begin with a simpler case where we omit the diffusion on $u_\eps$, arriving to a system of ordinary differential equations which can be solved nearly explicitly. This allows us to introduce the tools which are used in Section~\ref{sec:full} where we state and prove the concentration effect. Several related questions are detailed in appendices: the particular case of traveling waves, and some Sobolev regularity results.

\section{The problem without diffusion}
\label{sec:nodiff}

We begin with the simpler case where diffusion is ignored and where we can give a complete description while introducing the main tools for the general problem. We are reduced to a system of two differential equations with a parameter $x$, which solutions however behave as a front propagation in space, namely
\beq \begin{cases}
\p_t u_\eps (t,x) = \f 1 \eps u_\eps (v_\eps- \mu),  \qquad t \geq 0, \; x \in \R,
\\
\p_t v_\eps = - u_\eps v_\eps.
\end{cases}
\label{eq:1D}
\eeq

We define 
\[
Q(v)= v- \mu \ln v .
\]
We assume there are constants $Q_m$ and $Q_M $ such that 
\beq \begin{cases}
\text{For } x<0, \quad v_\eps^0(x) < \mu, \qquad \text{and for } x>0 \quad v_\eps^0(x) > \mu,
\\
 Q(v^0_\eps(x)) + \eps u^0_\eps \leq Q_M, \qquad Q(v^0_\eps(x))  > Q_m >Q(\mu),
\\
v_\eps^0 \to v^0 \quad \text{pointwise for } x \neq 0.
\end{cases} 
\label{nd:as1}
\eeq
In other words $v^0_\eps$ has  an initial  increasing discontinuity at $x=0$. For the cell density, we assume that  in the weak sense of measures,  as $\eps \to 0$, 
\beq \begin{cases}
u_\eps^0= \exp(\frac{\vp_\eps^0}{\eps}) \rightharpoonup \rho^0  \delta(x) \quad \text{in the weak sense of measures} , 
\\
\vp_\eps^0 \to \vp^0 \quad \text{in } C(\R), \qquad   \vp^0 <0 \text{ for } x \neq 0. 
\end{cases} 
\label{nd:as2}
\eeq

In this framework, we can describe the behavior of solutions as follows
\begin{theorem} Assume \eqref{nd:as1}-- \eqref{nd:as2}. The solution of \eqref{eq:1D} has limits $u_\eps  \rightharpoonup u$ (weak measures) and $v_\eps  \to v$ (strongly in $L^p_{loc}$) and, for $x>0$ there is a time $\tau(x)$ and $v^0_-(x)< \mu < v^0(x)$ such that
\\[2pt]
(i) $\; v(t,x) = v^0_-(x)$ for $t <\tau(x)$, \qquad $v(t,x) = v^0(x)$ for $t >\tau(x)$, 
\\[1pt]
(ii) $u(t,x) =  [\ln v^0(x) - \ln v^0_-(x)] \delta(t-\tau(x))$,
\\[1pt]
(iii)   $\tau(x)=-\frac{\vp^0(x)}{v^0(x)-\mu}$  with $\vp$ defined by \eqref{nd:eqlim}.
\\

For $x<0$, we have $v(t,x)= v^0(x)$. 
\label{th:nodiff}
\end{theorem}

\begin{proof}
A  remarkable property of the system \eqref{eq:1D} is the identity
\[
\p_t [ \eps u_\eps + v_\eps - \mu \ln v_\eps]= 0, 
\]
which implies 
\beq
\eps u_\eps + Q( v_\eps ) = Q(v_\eps^0(x)) + \eps u^0_\eps \leq Q_M,  \quad \forall x \in \R.
\label{nd:idQ}
\eeq
Using this inequality and the fact that $v_\eps $ decreases, we first conclude for $x<0$, $v_\eps- \mu <0$ and thus $u_\eps(t,x) \to 0$ as $\eps \to 0$, therefore $v_\eps(t,x) \to v^0(x)$   
\\

For $x>0$, we conclude from \eqref{nd:idQ} and assumption~\eqref{nd:as1},  that there are constants such that
\beq
v_m \leq  v_\eps (t,x) \leq v_M.
\eeq
  
 Also, integrating  the second equation of \eqref {eq:1D}, we have for all $T>0$, 
 \[
 \int_0^T u_\eps (t,x) dt  = \ln v_\eps^0(x) - \ln v_\eps(T,x) \leq C(T),
 \] 
 and since we expect that $u_\eps$ is a concentrated measure, we define $\vp_\eps = \eps \ln u_\eps$. It satisfies
 \beq \label{nd:phieps}
 \p_t \vp_\eps  =v_\eps -\mu  \quad \text{ is bounded in $t$ and $x$}.
 \eeq
 
 We can argue $x$ by $x$ and define as $\eps \to 0$ (after extraction of subsequences, but the uniqueness of the limit shows that it is the full family) the strong limits 
 \[
 v_\eps (t,x) \to v (t,x) , \qquad  \vp_\eps \to \vp(t,x).
 \]
We can also define the weak limit
 \[
  \qquad u_\eps (t,x) \rightharpoonup m (t,x)\geq 0 \in \cam^1(0,T) .
  \]
These limits satisfy the equations obtained passing to the limit in~\eqref{eq:1D} and~\eqref{nd:phieps}
 \beq \label{nd:eqlim}
   \begin{cases}
\p_t \vp(t,x) = v - \mu, \qquad  \vp(t,x)\leq 0
\\
\p_t \ln (v)  = -m (t,x), \qquad \text{supp}(m) \subset \{\vp(t)=0\}.
\\
Q(v(t,x)) = Q(v^0(x)). 
\end{cases} \eeq

Because of this last equality, $v(t,x)$ belongs to one of the two branches $v^0_-(x)< v^0_+(x):=v^0(x)$ of roots of $Q(v(t,x)) = Q(v^0(x))$. Therefore $v - \mu$ is away from $0$ and there is a first time $\tau(x)$ such that $\vp \big(\tau (x), x \big)=0$ and $\vp(t,x) <0$ for $t \neq \tau(x)$. This time $\tau(x)$ is also the jump  time from one branch to the other for $v$ as stated in (i).
\\

We can  also compute, from the above equation on $\ln v$,  and this gives (ii), namely
\[
u_\eps =- \p_t \ln v_\eps \rightharpoonup u=[\ln v_+(x) - \ln v_-(x)] \delta(t-\tau(x)).
\]

The time $\tau(x)$ is fully identified from the limiting system: By integrating the equation on $\vp(t,x)$, we find that
  \[
  \vp(t,x)=\left\{\begin{aligned} &\vp(0,x)+t\,(v^0(x)-\mu),\quad t<\tau(x),\\
  &\vp(0,x)+\tau(x) (v^0(x)-\mu)+(t-\tau(x))\,(v_-^0(x)-\mu),\quad t>\tau(x).\\
  \end{aligned}\right.
  \]
  For $x>0$, as $v^0_-(x)<\mu<v^0(x)$, $\vp(t,x)$ is strictly increasing in time for $t<\tau(x)$, and strictly decreasing for $t>\tau(x)$.

Furthermore by the second equation, we have that $\vp(\tau(x),x)=0$. Hence this characterizes $\tau(x)$ as
  \[
\tau(x)=-\frac{\vp^0(x)}{v^0(x)-\mu}.
\]
This gives (iii) and identifies completely the limiting solution $x$ by $x$.

\qed
\end{proof}

%
\section{The full problem}
\label{sec:full}

When diffusion is included, the previous analysis, which uses fundamentally $x$ by $x$ convergence, does not apply and the assumptions have to take into account the diffusion term. We also make more general assumptions. For the initial data, we assume that
\beq
 u^0_\eps >0, \qquad \eps u^0_\eps \leq C,  \qquad \int_\R u^0_\eps dx \leq C,
\label{assumption2}
\eeq
and, $\vp^0_\eps =\eps \ln u^0_\eps$ satisfies
\beq
 \eps \partial^2_{xx} \vp^0_\eps \leq C, \qquad  |\p_x \vp^0_\eps| \leq C , \qquad  \vp^0_\eps (x) \geq - C(1+ |x|),
\label{assumption3}
\eeq
\beq
 \eps  \, | \partial^2_{xx} \ln v^0_\eps | \leq C, \qquad |\partial_x \ln v^0_\eps|+|\ln v^0_\eps| \leq C.
\label{assumption1}
\eeq
Then, recalling the definition $Q(v)= v-  \mu \ln v$, we also use 
\beq
 w= \ln v, \qquad \wt Q(w) = Q(v) = e^w-\mu w,
\eeq
and we can define uniquely two smooth branches of initial data $w^0_{\pm,\eps}$ by 
\beq \label{full:Q0}
\wt Q(w^0_\eps) = \wt Q(w^0_{-,\eps} )= \wt Q(w^0_{+,\eps} ), \qquad w^0_{-,\eps}  \leq \ln \mu   \leq w^0_{+,\eps} .
\eeq
Note that assumption \eqref{assumption1} provides corresponding bounds on $w^0_\eps$ and $w^0_{\pm,\eps}$,  which are
\beq
|w^0_\eps|+|w^0_{\pm,\eps}|\leq C, \quad |\partial_x w^0_\eps|+|\partial_x w^0_{\pm,\eps}|\leq C,\quad \eps\,|\partial_{xx}^2 w^0_\eps|+ \eps |\partial_{xx}^2 w^0_{\pm,\eps}|\leq C. \label{assumptionweps}
\eeq
And we also use, for an unessential result, the stronger condition
\beq
 \eps \partial^2_{xx}  (w^0_\eps -w^0_{-,\eps}) + u^0_\eps \leq C.
\label{assumption4}
\eeq
Finally, we use the notation $|W|_-= \max(0, -W)$.

Our main theorem now reads
%
\begin{theorem} \label{th:diff}
Assume that $u_\eps^0\in C^2\cap L^1(\R)$ and that assumptions \eqref{assumption2}--\eqref{assumption1} hold. The solution of \eqref{eq:full} has limits $u_\eps  \rightharpoonup u$ (weak measures) and $v_\eps  \to v$ (strongly in $L^p_{loc}$) and, there is a time $\tau(x)$ and $v^0_-(x)< \mu < v^0(x)$ such that
\\[2pt]
(i) $\; v(t,x) = v^0_-(x)$ for $t <\tau(x)$, \qquad $v(t,x) = v^0(x)$ for $t >\tau(x)$, 
\\[1pt]
(ii) $u(t,x) =  [\ln v^0(x) - \ln v^0_-(x)] \delta(t-\tau(x))$.
\\
\end{theorem}
Compared to the case without diffusion in Theorem~\ref{th:nodiff}, there is no simple explicit formula for the jump time $\tau(x)$. However, here it is also characterized by $\vp(\tau(x),x)=0$ for the solution of an Eikonal equation. 
\\

The end of this section is devoted to the proof of Theorem~\ref{th:nodiff}.
\\

Our analysis is based on the  observation that $\p_t  \ln v_\eps=-u_\eps $, thus, from \eqref{eq:full},  we get the identity
\[
\p_t [\eps  u_\eps   + v_\eps   - \mu  \ln v_\eps + \eps^2 \partial^2_{xx} \ln v_\eps] =0,
\]
and consequently
\[
\eps \p_t  \ln v_\eps  - \eps^2 \partial^2_{xx} \ln v_\eps  - Q(v_\eps)  = - Q(v^0_\eps) - \eps  u^0_\eps -\eps^2 \partial^2_{xx} \ln v^0_\eps ,
\]
which we can write in terms of $w_\eps$ as 
\beq
\eps \p_t w_\eps  - \eps^2 \partial^2_{xx} w_\eps  - \wt Q(w_\eps)  = - \wt Q(w^0_\eps) - \eps  u^0_\eps - \eps^2 \partial^2_{xx} w^0_\eps.
\label{full:fund2}
\eeq
Notice that this is a monostable equation of Fisher/KPP type, where the steady states depend on~$x$.
\paragraph{A priori estimates on $v_\eps$.}

\begin{lemma} \label{lm:wbd} The inequalities hold
\beq
0< C  \leq  w_\eps (t,x)  \leq w^0_{+,\eps},
\label{boundweps0}
\eeq
and for any $R$,  there is a constant $C_R$ such that,
\beq
\int_{|x|\leq R} |\wt Q(w_\eps) - \wt Q(w^0_{\eps} ) | \ind{w_\eps - w^0_{-,\eps}\leq 0 }  \,dx \leq C_R\,\eps, \qquad \forall t\geq 0.
\label{boundweps1}
\eeq 
  Finally, with the additional assumption \eqref{assumption4}, we have $-C \sqrt{\eps} + w^0_{-,\eps} \leq  w_\eps (t,x) $.
\end{lemma}
\begin{proof}
For the first statement,  on the one hand, we note that we necessarily have that for any $x$, $w_\eps^0(x)=w^0_{-,\eps}(x)$ or $w_\eps^0(x)=w^0_{+,\eps}(x)$. In both cases, that implies that $w_\eps^0(x)\leq w^0_{+,\eps}(x)$. Since $v_\eps$ is non-increasing in time, so is $w_\eps$ and
\[
w_\eps(t,x)\leq w^0_\eps(x)\leq w^0_{+,\eps}(x).
 \]
On the other hand, as $\p_t w_\eps \leq 0$, we can deduce from \eqref{full:fund2},  that
\beq \label{full:negbr}
-\eps^2 \partial^2_{xx} (w_\eps - w^0_{-,\eps} )  + [ \wt Q(w^0_{-,\eps} )-\wt Q(w_\eps) ] \geq- \eps^2 \partial^2_{xx}  (w^0_\eps -w^0_{-,\eps}) - \eps u^0_\eps.
\eeq
Using the maximum principle and assumptions \eqref{assumption2}, \eqref{assumption1}, we conclude that at a minimum value of $w_\eps - w^0_{-,\eps}$, the quantity  $\wt Q(w^0_{-,\eps} )-\wt Q(w_\eps)$ is controlled from below and  the lower bound on $w_\eps$ follows.
\\

Furthermore, from the usual convex inequalities, we also observe that
\beq
- \eps^2  \partial^2_{xx} (w_\eps - w^0_{-,\eps} )_-  +  [\wt Q(w_\eps) - \wt Q(w^0_{-,\eps} )]  \ind{w_\eps - w^0_{-,\eps} <0}   \leq \eps\,u^0_\eps+\eps^2 |\partial^2_{xx}  (w^0_\eps -w^0_{-,\eps})|.\label{eqw-}
\eeq
 Because the set $\{W \, s.t. \,W - w^0_{-,\eps} <0 \}$ lies in the decreasing branch of $\wt Q$, the quantity 
 $ [\wt Q(w_\eps) - \wt Q(w^0_{-,\eps} )]  \ind{w_\eps - w^0_{-,\eps} <0}  $ is positive.  We integrate against some smooth non-negative $\psi$ with $\psi=1$ on $B(0,R)$ and~$\psi$ compactly supported in $B(0,2R)$.  Since we have already proved that $w_\eps - w^0_{-,\eps}$ is bounded, we obtain the estimate
 \[
\int_{|x|\leq R} |\wt Q(w_\eps) - \wt Q(w^0_{\eps} ) | \ind{w_\eps - w^0_{-,\eps}\leq 0 }  \,dx \leq C_R\,\eps^2+\eps\,\int_{|x|\leq R} u_\eps^0\,dx+\eps^2\,\int_{|x|\leq R}|\partial^2_{xx}  (w^0_\eps -w^0_{-,\eps})| \,dx.
\]
Using assumption~\eqref{assumption2} and assumption \eqref{assumption1} concludes the second point of the lemma.

We may also use the specific assumption \eqref{assumption4} in~\eqref{eqw-}, we obtain that
\[
- \eps^2  \partial^2_{xx} (w_\eps - w^0_{-,\eps} )_-  +  [\wt Q(w_\eps) - \wt Q(w^0_{-,\eps} )]  \ind{w_\eps - w^0_{-,\eps} <0}   \leq C \eps.
\]
Recalling that $ [\wt Q(w_\eps) - \wt Q(w^0_{-,\eps} )]  \ind{w_\eps - w^0_{-,\eps} <0}  $ is positive, we conclude that
 \[
 Q(w_\eps) - \wt Q(w^0_{-,\eps} )\ind{w_\eps - w^0_{-,\eps} <0}  \leq C \eps, 
\]
and thus the third statement of Lemma \ref{lm:wbd} holds, namely $w_\eps - w_{-,\eps}\geq -C\,  \sqrt \eps$.
\\
\qed
\end{proof} 
%
\paragraph{Concentration  dynamics of $u_\eps$.}
We turn to the study of $u_\eps$ and begin with some simple estimates.
\\
$\bullet$ Since $u_\eps=-\p_t w_\eps$, and using the   bound~\eqref{boundweps0},  we find
\beq \label{est:utimeint}
\sup_x \int_0^T u_\eps (t,x ) dt = \sup_x [ w_\eps^0(x)-w_\eps(x,T) ] \leq C(T),
\eeq
and thus, integrating the equation on $u_\eps$, we also get
\beq \label{bd:epsu}
\eps\, \int_{\R} u_\eps(t,x)\,dx =\eps\,\int_{\R} u^0_\eps (x)\,dx  + \int_{\R}\int_0^t u_\eps(s,x) (v_\eps(s,x)-\mu) ds\,dx \leq C(t).
\eeq

\bigskip

$\bullet$ Next, we use the Hopf-Cole transform
\[
\vp_\eps (t,x) =  \eps \ln u_\eps(t,x).
\]
As usual, we compute that $\vp_\eps$ satisfies the Eikonal equation
\beq \label{eq:phieps}
\p_t \vp_\eps =\eps \partial^2_{xx} \vp_\eps  +|\partial_x \vp_\eps |^2 + v_\eps - \mu .
\eeq

We are going to show some uniform bounds on $\vp_\eps$.
\begin{lemma}\label {lm:bdphieps}
We have
\beq \label{bd:dtphi}
\p_t \vp_\eps (t,x) \leq  C,  \qquad \forall x \in \R, \; t\geq 0
\eeq
\beq \label{bd:dxphi}
\| \partial_x \vp_\eps (t)\|_{L^\infty(\R)} \leq C, \qquad   \| \p_t \vp_\eps \|_{L^p_{t,x}} \leq C(p), \quad \forall p\in [1, \infty),
\eeq
\beq \label{bd:phiup}
-C(t) (1+|x|)  \leq \vp_\eps(t,x) \leq 2\,\eps\,\ln\frac{1}{\eps}+C(t)\,\eps, \qquad   \forall x \in \R, \; t\geq 0 .
\eeq
\end{lemma}

\begin{proof} 
For the time derivative,  differentiating \eqref{eq:phieps} and using the equation on $v_\eps$, we find 
\[
\p_t (\p_t \vp_\eps) =\eps \partial^2_{xx} (\p_t \vp_\eps)  +2 \partial_x \vp_\eps  \partial_x (\p_t \vp_\eps)  - u_\eps v_\eps \leq \eps \partial^2_{xx} (\p_t \vp_\eps)  +2 \partial_x \vp_\eps  \partial_x (\p_t \vp_\eps),
\]
so that the maximum principle gives $\p_t \vp_\eps (t,x) \leq \max_x  \p_t \vp_\eps^0(x)$ which gives  \eqref{bd:dtphi} thanks to the assumption~\eqref{assumption3}. 
\\

Next, we  prove the Lipschitz bound. Consider any point $x_\eps$ that is a maximum in $x$ of $\partial_x \vp_\eps$ at any time $t$ (standard arguments apply if the maximum is not achieved, see \cite{CEL_hj, Barles1994}). Then $\partial_{xx}^2 \vp_\eps(x_\eps)=0$ and we conclude, still using~\eqref{bd:dtphi},  that
\beq \label{est:phiLip}
|\p_x\vp_\eps(t,x_\eps)|^2=\p_t\vp_\eps(t,x_\eps)+v_\eps-\mu\leq C.
\eeq
Once $\partial_x \vp_\eps\in L^\infty_{t,x}$ uniformly, standard parabolic estimates provide a uniform bound on $\partial_t \vp_\eps$ in $L^p_{t,x}$ for any $1 \leq p<\infty$.
\\

Finally, since $\vp_\eps$ is uniformly Lipschitz in $x$, let $x_\eps$ be a maximum of $\vp_\eps$, then
\[
\vp_\eps(t,x)\geq \max \vp_\eps(t,.)-C\,|x-x_\eps|,
\]
so that, using \eqref{bd:epsu},
\[
\frac{C}{\eps} \geq \int_\R u_\eps(t,x)\,dx\geq \int_\R e^{\max \vp_\eps(t,.)/\eps}\,e^{-C\,|x-x_\eps|/\eps}\,dx\geq \frac{\eps}{C}\, e^{\max \vp_\eps(t,.)/\eps}, 
\]
which proves the upper bound in~\eqref{bd:phiup}. The lower bound relies, as it is standard \cite{Barles1994,CEL_hj,BMP09}, on the construction of a sub-solution. Here one can immediately check that $-C(t+1) - \frac{C x^2}{\sqrt {1+ |x|^2}}$ will work.
\qed
\end{proof}

\paragraph{Compactness of  $v_\eps$.}

We introduce the quantity $\Phi(x,w)$, defined up to a constant, by 
$$
\Phi_w (x,w)= |\wt Q(w^0_\eps (x) )- \wt Q(w)| \geq 0.
$$
\begin{lemma} \label{lm:vepscompact}
With assumptions~\eqref{assumption1}--\eqref{assumption2}, we have
\[
\sup_{x \in \R, \, 0\leq t \leq T} | \p_x w_\eps | \leq \f {C_{T}} \eps , \qquad  \int_0^T \! \! \int_{|x| \leq R}  \big| \p_x  \Phi(x,w_\eps(t,x) ) \big| \leq C_{T, R} .
\]
Consequently, by monotonicity of $\Phi$ in $w$,  $v_\eps$ is  locally compact in $L^p( (0, \infty)\times \R)$ for any $1\leq p<\infty$. 
\end{lemma}

\begin{remark}
 It is also possible to conclude from this lemma that $w_\eps$ is uniformly bounded in $L^1([0,\ T],\;W^{\theta,1}(\R))$ for some $\theta>0$; see Appendix~\ref{ap:sobolev}.
\end{remark}
\begin{proof}
First of all, calculate
\[
\p_t\partial_x w_\eps=-\partial_x u_\eps =- \frac{\partial_x \vp_\eps}{\eps}\,u_\eps,
\]
which yields, from the Lipschitz bound on $\vp_\eps$ in \eqref{est:phiLip} and the estimate~\eqref{est:utimeint},
\[
|\partial_x w_\eps(t,x)|\leq |\partial_x w_\eps^0(x)|+\frac{C}{\eps}\,\int_0^t u_\eps(s,x)\,ds \leq \frac{C(t)}{\eps}.
\]

Next, we write $ \p_x [\Phi(x,w_\eps (t,x))]  = \p_x \Phi(x,w_\eps ) + [\wt Q(w_\eps^0) - \wt Q(w_\eps)] \p_x w_\eps $. Since $\p_x \Phi(x,w_\eps ) $ is bounded in $L^\infty$, and thanks to the second bound in Lemma~\ref{lm:wbd}, we conclude that 
\begin{equation}
 \int_0^T \! \! \int_{|x| \leq R}  |\p_x \Phi(x,w_\eps ) | \leq C_{T, R} + \eps \sup_{|x| \leq R, \, 0\leq t \leq T} | \p_x w_\eps |   \int_0^T \! \! \int_{|x| \leq R}  | \frac{Q(v^0_\eps) - Q(v_\eps)}{\eps} | \leq C_{T, R} .\label{BVPhi}
\end{equation}
Since $\p_t w_\eps \leq 0$ and $w_\eps$ bounded provide us with compactness in time, we conclude that  $\Phi(x,w_\eps )$ is compact and thus converges a.e. By monotonicity of $\Phi$ in $w_\eps$, we also conclude that $w_\eps $ converges a.e.
\qed
\end{proof}

\paragraph{Convergence as $\eps \to 0$.}  We are now ready to study the limit as $\eps$ vanishes.
\\
$\bullet$ The bounds in Lemma~\ref {lm:bdphieps} show that $\vp_\eps$ is locally compact in $C(\R_+\times \R)$ and hence, after extraction of a subsequence that we still denote by $\eps$, there exists $\vp$  which is Lipschitz in space and with time derivatives in $L^p$ such that
\[
\|\vp_\eps - \vp\|_{L^\infty ((0,T)\times (-R, R))}\to 0,\quad \mbox{as}\ \eps \to 0, \qquad \forall T>0, \; R >0.
\]
We note that
\beq
 -C(1+t) -C |x| \leq \vp(t,x) \leq 0.
\eeq
$\bullet$ From Lemma \ref{lm:vepscompact}, we also conclude that  $v_\eps$ converges locally;  for any $1\leq p<\infty$,
\[
v_\eps \to v,  \quad w_\eps \to w=\ln v  \quad \text{in} \quad L^p((0,T) \times (-R,R)), \qquad \forall T>0, \; R>0.
\]
$\bullet$ From the bound \eqref{est:utimeint}, we can also extract a subsequence such that, in the weak sense of measures, 
\[
u_\eps \rightharpoonup u,  \quad \text{in} \quad \mathcal M ((0,T) \times (-R,R)), \qquad \forall T>0, \; R>0.
\]
We may pass to the limit, in distributional sense, in Equations~\eqref{eq:full} and get, as $\eps \to 0$,
\[
u_\eps v_\eps \rightharpoonup  u \mu , \qquad \p_t w =- u \qquad \p_t v= \text{w-lim } (-u_\eps v_\eps) =- u \mu ,
\]
which expresses that the concentration of the measure $u_\eps$ is exactly at the point where $v_\eps=\mu$. 
And from Lemma \ref{lm:wbd}, we know that
\[
\wt Q(w)= \wt Q(w^0), \quad i.e.,  \quad   w(t,x) = w^0_-(x) \quad or \quad w^0_+(x).
\]

Since $w$ is non-decreasing in time, we conclude that, for all $x\in \mathcal S$, a subset of $\R$ where $w^0=w^0_-$,  there is a unique time $\tau(x)$ such that $w$ jumps from $w^0_-(x) $ to $w^0_+(x) $ (with $\tau(x)= \infty$ for $x<0$), and 
\[
u(t,x) =[w^0_+(x)- w^0_-(x)] \delta(t-\tau(x)) \ind{x \in \mathcal S},  \qquad \vp(\tau(x), x) = 0.
\]

\paragraph{Open questions.} Uniqueness for the limit problem, which we proved when diffusion is ignored (Section~\ref{sec:nodiff}, is an open question in full generality. In particular it seems hard to determine more properties about the set $\mathcal S$, which depends on the initial data. In the monotone case when $w^0(x)=w^0_-(x)$ for $x<0$ and $w^0(x)=w^0_+(x)$ for $x>0$ with $u^0= [[ w^0]] \delta(x)$, one can expect $\tau(x)$ be invertible and to obtain a front located at some value $x=X(t)$. 

%
%
%
\medskip

\noindent{\bf Acknowledgments.} 
The authors would like to thank the Institut Henri Poincar\'e for its support and hospitality during the program ``Mathematical modeling of organization in living matter''.
The authors also thank similarly the Isaac Newton Institute for Mathematical Sciences, Cambridge, for support and hospitality during the program "Frontiers in kinetic theory: connecting microscopic to macroscopic scales - KineCon 2022" where work on this paper was undertaken. This work was supported by EPSRC grant no EP/K032208/1. B.P. has received funding from the European Research Council (ERC) under the European Union's Horizon 2020 research and innovation programme (grant agreement No 740623). P.~E. Jabin is partially supported by NSF DMS Grants DMS-2049020, DMS-2205694, and DMS-2219397.

\appendix

\section{Traveling wave} 

Traveling waves are an intuitive way to understand, in a very particular case, the general behavior of system~\eqref{eq:full}. considering solutions of the form $u_\eps(x-\sg t)$,  $v_\eps(x-\sg t)$,  $w_\eps= \ln v_\eps$. Recalling the notation $\wt Q( w ) =e^{w}- \mu w$, we arrive at an equation on the single quantity $w_\eps(y)$, 
\[
-\sg \eps w_\eps' - \eps^2 w_\eps''=\wt Q( w_\eps)- A,
\] 
with the conditions at infinity
\[
w_\eps(-\infty)=w_- < \ln \mu, \qquad w_\eps(+ \infty)= w_+ > \ln \mu, \qquad A= \wt Q( w_-)
=\wt Q( w_+).
\]
This is just a Fisher/KPP monostable equation with $w_+$ the unstable state and we know from the general theory~\cite{LamLou} that there is a traveling wave with minimal speed $\sg_*$ which is characterized by the property of a double root for the polynomial
\[
-\sg \eps \lb -\eps^2 \lb^2 = \wt Q'(w_+),
\]
that is $\sg_*= 2 \sqrt{ \wt Q'(w_+)}$.
In our analysis, this value $\sg_*$  also appears in the limit  of Equation~\eqref{eq:phieps}, that is the  Eikonal equation 
\[
\p_t \vp  = |\partial_x \vp |^2 + v - \mu ,
\]
which for the traveling wave problem generates a solution $\vp(x-\sg t)$ with 
\[
\sg \vp' (y)= | \vp'(y)|^2 + v_\pm(y) -\mu, 
\]
where $v_\pm(y)=v_-$ for $y<0$ and $v_\pm(y)=v_+$ for $y>0$.
The limiting minimal speed  traveling wave solution is 
\[
\vp(y)= 
\begin{cases}
p_- y <0\qquad  &\text{for} \quad y <0,
\\
 p_+ y <0  \qquad & \text{for} \quad  y >0,
\end{cases}
\]
and $p_+$ is the double root of the polynomial $-\sg_*   \lb - 2 \lb^2 = \wt Q'(w_+)= v_+- \mu$. This approach based on the concentration as a Dirac measure of $u_\eps$  differs (but is restricted to one dimension) from the general front propagation theory in \cite{BES90,BS94} based on the quantity $v_\eps$. 
\commentout{
\section{Energy} 

Equation \eqref{full:fund2} comes with an energy 
\[
E(t):= \int [\eps^2 |\nabla w_\eps|^2+  \big(Q(w^0_\eps(x)) + \eps u^0(x) +\eps^2 \Delta w^0_\eps \big) w_\eps- e^{w_\eps} + \mu \frac {w_\eps^2} 2] dx.
\] 
\red{The difficulty here is that it is not  finite  with the only information that $Q(w^0_\eps(x)) - Q(w_\eps(t,x)) \in L^1$. P-E: Should we just require $u^0_\eps$ to be smooth? That should be enough, together with assumption (9) to propagate $\eps$-dependent smoothness on $v_\eps$ and allow to write everything on the energy.}

We immediately compute 
\[
\frac {d E(t)}{dt}= - \frac 1 \eps \int \big[ - \eps^2 \Delta w_\eps  - \wt Q(w_\eps) + \wt Q(w^0_\eps) + \eps  u^0_\eps + \eps^2 \Delta w^0_\eps \big]^2 dx .
\]
The energy decay  tells us  that $E(t)$ is bounded and thus $\eps^2 \int |\nabla w_\eps|^2$ and also
\[
\int_0^T \int   \big[ - \eps^2 \Delta w_\eps  - \wt Q(w_\eps) + \wt Q(w^0_\eps) + \eps  w^0_\eps + \eps^2 \Delta w^0_\eps \big] w = 0(\eps).
\]
This is related to compactness of $w_\eps$. indeed, using Lemma~\ref{lm:wbd}, we find the estimate 
\[
\eps  \int_0^T \int |\nabla w_\eps|^2  \leq C.
\]
Then, being given $\Psi$, we compute
\[
\big( \int | \nabla \Psi (x,w)| dx \big)^2 \leq \int_0^T \int |\nabla w_\eps|^2  \int_0^T \int | \Psi _w(x,w)|^2 dx + C
\]
And the choice
\[
\Psi _w(x,w) = | Q(v^0_\eps)- Q(w)|^{1/2}  
\]
gives that  $\Psi (x,w_\eps)$ is compact and since $\Psi (x,w)|$ is an increasing function, concludes immediately that $w_\eps$ is compact. 
}
\section{A Sobolev estimate} 
\label{ap:sobolev}

We may use the bound \eqref{BVPhi} to obtain Sobolev regularity on $v_\eps$ by controlling
\begin{equation}
\sup_{|h|\leq 1} \int_{0}^T  \! \!  \int_{|x|\leq R} \frac{|v_\eps(x+h)-v_\eps(x)|}{|h|^{\theta}}\,dx\,dt,\label{normsobolev}
\end{equation}
for some appropriate value of $\theta$.

This requires to be a bit more precise on the set where the initial data $v_\eps$ crosses $\mu$. Specifically, we assume that there exists some constants $C>0$ and $\kappa>0$ such that for any $\delta>0$
  \begin{equation}
|\{x,\;|v_\eps^0(x)-\mu|\leq \delta\}|\leq C\,\delta^\kappa.\label{v=mu}
  \end{equation}

Observe that when $|h|\leq \eps^{1/(1-\theta)}$ then by the Lipschitz bound on $v_\eps$ (which follows immediately from that on $w_\eps$), then 
\[
 \int_{0}^T  \! \!  \int_{|x|\leq R} \frac{|v_\eps(x+h)-v_\eps(x)|}{|h|^{\theta}}\,dx\,dt\leq C_{T,R}\,\|\partial_x v_\eps\|_{L^\infty}\,|h|^{1-\theta}\ \leq C_R, 
 \]
 so that we can limit ourselves to $h\geq \eps^{1/(1-\theta)}$.
 
For some $\alpha >0$ which we later relate to $\theta$ and some constant $C$, denote
\[
\begin{split}
  &\Omega_-=\{(t,x)\in [0,\ T]\times B(0,R),\;w^0_{-,\eps}-|h|^{\alpha}\leq w_\eps(t,x)\leq \ln\mu-|h|^{\alpha/2}\},
  \\
  &\Omega_+=\{(t,x)\in [0,\ T]\times B(0,R),\;w_\eps(t,x)\geq \ln\mu+|h|^{\alpha/2}\},\\
  & \Omega_0=\{(t,x)\in [0,\ T]\times B(0,R),\; \ln\mu-|h|^{\alpha/2}\leq w_\eps(t,x)\leq \ln\mu+|h|^{\alpha/2} \ \text{or}\ w_\eps(t,x)\leq w^0_{-,\eps}-|h|^\alpha\}.
  \end{split}
\]
Observe that when $(t,x)\in \Omega_0$ then 
\[
\begin{split}
  &  \sup_{\eps^{1/(1-\theta)}\leq |h|\leq 1} \int_{(t,x)\in \Omega_0\ \text{or}\ (t,x+h)\in \Omega_0} \frac{|v_\eps(t,x+h)-v_\eps(t,x)|}{|h|^{\theta}}\,dx\,dt\\
  &\qquad\leq 2\,\sup_{\eps^{1/(1-\theta)}\leq |h|\leq 1}\frac{\|v_\eps\|}{|h|^\theta}\,\int_{(t,x)\in \Omega_0\ \text{or}\ (t,x+h)\in \Omega_0}dx dt\\
\end{split}
\]
We note that if $w_\eps(t,x)\leq w^0_{-,\eps}(x)-|h|^{\alpha}$ then
\[
|\wt Q(w_\eps)-\wt Q(w_\eps^0)|\geq |h|^\alpha/C.
\]
Similarly if $\ln\mu-|h|^{\alpha/2}\leq w_\eps(t,x)\leq \ln\mu+|h|^{\alpha/2}$ but 
$w_\eps^0(x)<\ln\mu-2\,|h|^{\alpha/2}$ or $w_\eps^0(x)>\ln\mu+2|h|^{\alpha/2}$, then we have again
\[
|\wt Q(w_\eps)-\wt Q(w_\eps^0)|\geq |h|^\alpha/C,
\]
as $\wt Q(w)$ has a minimum at $w=\ln\mu$ but is strictly convex.

Hence by \eqref{v=mu}
  \[\begin{split}
  |\Omega_0|&\leq \frac{C}{|h|^\alpha}\,\int_0^T  \! \!  \int_{B(0,R)} |\wt Q(w_\eps)-\wt Q(w_\eps^0)|\,dx\,dt+\{x,\;|w_\eps^0(x)-\ln \mu|\leq 2\,|h|^{\alpha/2}\}\\
  & \leq \frac{C}{|h|^\alpha}\,\int_0^T  \! \!  \int_{B(0,R)} |\wt Q(w_\eps)-\wt Q(w_\eps^0)|\,dx\,dt+C\,|h|^{\kappa\,\alpha/2}.
\end{split}
\]
We therefore obtain that
\[
\begin{split}
&\sup_{\eps^{1/(1-\theta)}\leq |h|\leq 1} \int_{(t,x)\in \Omega_0\ \text{or}\ (t,x+h)\in \Omega_0} \frac{|v_\eps(t,x+h)-v_\eps(t,x)|}{|h|^{\theta}}\,dx\,dt\\
&\qquad\leq \sup_{\eps^{1/(1-\theta)}\leq |h|\leq 1} \frac{C}{|h|^\theta}\,\frac{1}{|h|^\alpha}\,\int_0^T  \! \! \int_{B(0,R)} |\wt Q(w_\eps)-\wt Q(w_\eps^0)|\,dx\,dt+C\,\frac{|h|^{\kappa\,\alpha/2}}{|h|^\theta}\leq C_{T,R},
\end{split}
\]
by Lemma \ref{lm:wbd} and provided that $\kappa\,\alpha/2\geq \theta$ leading us to take $\alpha=2\,\theta/\kappa$ and $(\alpha+\theta)/(1-\theta)=(1+2/\kappa)\,\theta/(1-\theta)\leq 1$. It is always possible to satisfy these inequalities provided that $\theta/(1-\theta)\leq \frac{1}{1+2/\kappa}$.

We can consequently exclude the case where $(t,x)\in \Omega_0\ \text{or}\ (t,x+h)\in \Omega_0$ when bounding \eqref{normsobolev}.

The $BV$ bound~\eqref{BVPhi} also directly controls the case where $(t,x)\in \Omega_-\ \text{and}\ (t,x+h)\in \Omega_+$ (or vice-versa). Indeed in that case, we necessarily have that $w_{+,\eps}^0\geq w_{-,\eps}^0+2\,|h|^{\alpha/2}$ and $\partial_w \Phi(x,w_\eps)=\wt Q(w_\eps^0)-\wt Q(w_\eps)$ has a sign between $w_{-,\eps}$ and $w_{+,\eps}$ so that there exists a constant $C$ s.t. (again at least locally)
\[
|\Phi(x,w^0_{-,\eps})-\Phi(x,w^0_{+,\eps})|\geq \frac{|h|^{3\,\alpha/2}}{C}.
\]
By our definitions of $\Omega_-$ and $\Omega_+$ and taking $C$ large enough, this implies that
\[
|\Phi(x,w_\eps(t,x+h))-\Phi(x,w_\eps(t,x))|\geq \frac{|h|^{3\,\alpha/2}}{C}\quad\text{if}\ (t,x)\in \Omega_-\ \text{and}\ (t,x+h)\in \Omega_+.
\]
Therefore
\[
\begin{split}
&  \int_{(t,x)\in \Omega_-\ \text{and}\ (t,x+h)\in \Omega_+} \frac{|v_\eps(t,x+h)-v_\eps(t,x)|}{|h|^{\theta}}\,dx\,dt\leq 2\,\frac{\|v_\eps\|}{|h|^{\theta}}\,\int_{(t,x)\in \Omega_-\ \text{and}\ (t,x+h)\in \Omega_+}dx\\
  &\qquad \leq \frac{C}{|h|^{\theta+3\alpha/2}}\,\int_0^T  \! \!  \int_{B(0,R)} |\Phi(x,w_\eps(t,x+h))-\Phi(x,w_\eps(t,x))|\,dx\,dt\leq C_{T,R}\, |h|^{1-\theta-3\alpha/2},
\end{split}
\]
by~\eqref{BVPhi}. This is bounded as long as $\theta+3\alpha/2=\theta\,(1+3/\kappa)\leq 1$.

We are finally able to focus on the case where for example both $(t,x)\in \Omega_-$ and $(t,x+h)\in \Omega_-$. Note again that $\partial_w \Phi(x,w_\eps)=\wt Q(w_\eps^0)-\wt Q(w_\eps)$ vanishes once with a change of sign  at $w_\eps=w_{-,\eps}$ for $(t,x)\in \Omega_-$. Therefore $w\to \Phi(x,w)$ is injective on $w\in [w^0_{-,\eps},\; \ln \mu-|h|^{\alpha/2}]$ for some constant $C$ and 
\[
|w_1-w_2|^3\leq C\,|\Phi(x,w_1)-\Phi(x,w_2)|,\quad w_1,\,w_2\in [w^0_{-,\eps},\;\ln\mu-|h|^{\alpha/2}].  
\]
Since $w_\eps\geq w^0_{-,\eps}-|h|^\alpha$ on $\Omega_-$, this implies that for  both $(t,x)\in \Omega_-$ and $(t,x+h)\in \Omega_-$, we have
\[
|w_\eps(t,x)-w_\eps(t,x+h)|^3 \leq C\,|\Phi(x,w_\eps(t,x))-\Phi(x,w_\eps(t,x+h))|+C\,\eps^\alpha.  
\]
By a straightforward H\"older inequality, we get
\[
\begin{split}
  &\int_{(t,x)\in \Omega_-\ \text{and}\ (t,x+h)\in \Omega_+} \frac{|v_\eps(t,x+h)-v_\eps(t,x)|}{|h|^{\theta}}\,dx\,dt\\
  &\quad\leq C_{T,R}\,\left(\int_{(t,x)\in \Omega_-\ \text{and}\ (t,x+h)\in \Omega_-} \frac{|w_\eps(t,x+h)-w_\eps(t,x)|^3}{|h|^{3\theta}}\,dx\,dt\right)^{1/2}\\
 &\quad\leq C_{T,R}\,\left(\int_{(t,x)\in \Omega_-\ \text{and}\ (t,x+h)\in \Omega_-} \frac{|\Phi(x,w_\eps(t,x))-\Phi(x,w_\eps(t,x+h))|+|h|^\alpha}{|h|^{3\theta}}\,dx\,dt\right)^{1/2}\leq C_{T,R}, 
\end{split}
\]
by ~\eqref{BVPhi} again, provided that $3\,\theta\leq 1$ and $\alpha\geq 3\,\theta$. Since we chose $\alpha=2\,\theta/\kappa$, this last inequality holds provided that $\kappa\leq 2/3$, which we can always impose.

To summarize, we have obtained the desired bound~\eqref{normsobolev} provided that $\theta\leq 1/3$ and $\theta\leq 1/(1+3/\kappa)$ together with $\theta/(1-\theta)\leq 1/(1+2/\kappa)$.

\bibliographystyle{plain}
\bibliography{biblio}


\end{document}